\documentclass{amsart}
\usepackage{amsfonts}
\usepackage{amscd}
\usepackage{amssymb}

\numberwithin{equation}{section}

\usepackage{datetime}

\theoremstyle{plain}

\newtheorem{theorem}{Theorem}[section]
\newtheorem{proposition}[theorem]{Proposition}

\theoremstyle{definition}

\newtheorem{remark}[theorem]{Remark}
\newtheorem*{setup}{Context and basic assumption}
\newtheorem*{assumptions}{Possible additional assumptions}
\newtheorem*{summaryofresults}{Summary of results}

\newcommand{\R}{{\mathbb R}}
\newcommand{\Z}{{\mathbb Z}}
\newcommand{\C}{{\mathbb C}}
\newcommand{\N}{{\mathbb N}}
\newcommand{\T}{{\mathbb T}}
\newcommand{\F}{{\mathcal F}}
\newcommand{\D}{{\mathcal D}}
\newcommand{\Ddual}{\D^\prime}
\newcommand{\dist}{{\Xi}}
\newcommand{\Lag}{{\mathcal L}}

\newcommand{\eig}{{\varepsilon}}

\newcommand{\supp}{{\operatorname {supp}\,}}
\newcommand{\cl}{{\operatorname {cl}\,}}

\newcommand{\dirprod}{{\prod_{\lambda\in\Lambda}L_\lambda}}

\newcommand{\familyoffunctions}{{\Gamma}}
\newcommand{\seminormfunction}{{\gamma}}

\renewcommand{\S}{{\mathcal S}}

\begin{document}
\title[Local spectral radius formulas for unbounded operators]{{Local spectral radius formulas for a class of unbounded operators on Banach spaces}}
\author{Nils Byrial Andersen}
\address{Department of Mathematics,
Aarhus University,
Ny Munkegade 118,
Building 1530,
DK-8000 Aarhus C,
Denmark}
\email{byrial@imf.au.dk}
\author{Marcel de Jeu}
\address{Mathematical Institute,
Leiden University,
P.O. Box 9512,
2300 RA Leiden,
The Netherlands}
\email{mdejeu@math.leidenuniv.nl}

\begin{abstract}We exhibit unbounded operators on Banach spaces having the single-valued extension property, for which the local spectrum at suitable points can be determined, and for which a local spectral radius formula holds, analogous to that for a bounded operator on a Banach space with the single-valued extension property. Such an operator can occur as (an extension of) a differential operator which, roughly speaking, can be diagonalised on its domain of smooth test functions via a discrete transform which is an isomorphism of topological vector spaces between the domain, in its own topology, and a sequence space. We give examples (constant coefficient differential operators on the $d$-torus, Jacobi operators, the Hermite operator, Laguerre operators) and indicate further perspectives.
\end{abstract}
\maketitle

\section{Introduction and results}\label{sec:intro}

We are concerned with a class of unbounded operators on Banach spaces for which questions concerning the single-valued extension property, local spectra and local spectral radius formulas can be answered satisfactorily. The operators in question are, roughly speaking, those operators that can be diagonalised on their domain via a discrete transform, with very good convergence results for the inverse transform, or extensions of such operators. Such operators are not uncommon in the study of differential operators, on Euclidean spaces or on compact symmetric spaces, and it is for this reason that we hope that our results are not only a further step in the investigation of local spectral radius formulas for unbounded operators, but are of some interest in concrete cases as well, also beyond the examples we will present in this paper.

In this Section, we start with an introduction in Section~\ref{subsec:introductionandliterature}, also collecting what seems to be known about local spectral radius formulas for unbounded operators in general. To our knowledge, this is rather limited. We then continue in Section~\ref{subsec:frameworkandresults} with the general framework in which our results hold, and we summarise those results. The general framework is motivated by, in particular, differential operators, and we comment how these fit in. Section~\ref{sec:resultsandproofs} contains the precise statements and proofs. In Section~\ref{sec:examples} several examples are given, applying the results in Section~\ref{sec:resultsandproofs} along the lines as discussed in Section~\ref{subsec:frameworkandresults}. In some of these, estimates obtained in special function theory are used. In our final example, we indicate how further results can be obtained on basis of not so well known, but quite interesting, work of Zerner's.

\subsection{Introduction and existing literature}\label{subsec:introductionandliterature}
Let $X$ be a Banach space, $\D\subset X$ a linear subspace,
and $T: \D \to X$ a possibly unbounded linear operator, not necessarily closed.  A point $z_0\in\C$ is said to be in the local resolvent set of $x\in X$,
denoted by $\rho_T(x)$, if there is an open neighborhood $U$ of $z_0$
in $\C$, and an analytic function $\phi: U\to {\D}$, sending $z$ to
$\phi_z$, such that
\begin{equation}\label{eq:localresolvent}
(T-z)\phi_z=x\qquad(z\in U).
\end{equation}
The local spectrum $\sigma_T(x)$ of $T$ at $x$ is the complement of $\rho_T(x)$ in $\C$; it is a closed subset of $\C$.
If $x\in\D(T^\infty)=\cap_{n=0}^\infty\D(T^n)$, then the local spectral radius $r_T(x)$ of $T$ at $x$ is defined as
\begin{equation*}
r_T(x) = \limsup_{n\to \infty} \| T^n x\|^{1/n}
\end{equation*}
in the extended positive real numbers. The terminology is, to some extent, not entirely optimal, because the extended positive real number
\[
\sup\{|z| : z\in\sigma_T(x)\}
\]
could justifiably also be called the local spectral radius of $T$ at $x$, and perhaps even more so. If $X$ is complex, $\D=X$ and $T$ is bounded, then the two globally defined real numbers $\lim_{n\to \infty} \Vert T^n \Vert^{1/n}$ and $\max\{|z| : z\in\sigma_T\}$, where $\sigma_T$ is the spectrum of $T$ in the Banach algebra of bounded operators on $X$, are of course always equal by the classical spectral radius formula, so ambiguity in the terminology is unlikely to arise, but for their local counterparts things are different. We will now briefly discuss this.

If $T$ is closed, then
\begin{equation}\label{eq:inequalityforclosedoperators}
\sup\{|z| : z\in\sigma_T(x)\}\leq\limsup_{n\to\infty}\Vert T^n x\Vert^{1/n}\qquad(x\in\D(T^\infty))
\end{equation}
always holds in the extended positive real numbers.\footnote{This statement is similar in spirit to \cite[Proposition~IV.3.10]{Vasilescu}, where it is assumed that a local resolvent exists and is analytic on an open neighborhood of $\infty$ on the Riemann sphere. In that case, the hypothesis that $x\in\D$ is sufficient, and one concludes that actually $x\in\D(T^\infty)$. Moreover, the right hand side in \eqref{eq:inequalityforclosedoperators} is then necessarily finite.} Indeed, we may assume that the right hand side in \eqref{eq:inequalityforclosedoperators} is finite. In that case, for $|z|>r_T(x)$, the Neumann-type series
\begin{equation}\label{eq:Neumannseries}
- \sum_{n=0}^\infty \frac{T^n x}{z^{n+1}}
\end{equation}
converges in $X$, and from the fact that $T$ is closed it follows easily that the sum $\phi_z$, which clearly depends analytically on $z$ on $\{z\in\C : |z|>r_T(x)\}$, is in fact in $\D$, and satisfies $(T-z)\phi_z=x$. Hence $z\in\rho_T(x)$ for all $z>r_T(x)$, as desired.

The reverse inequality in \eqref{eq:inequalityforclosedoperators}, and hence equality, need not always hold, however, not even for globally defined bounded operators, and we will briefly review a few relevant results.
Recall that an operator $T$ with domain $\D$ is said to have the single-valued extension property if, for every
non-empty open subset $U\subset\C$, the only analytic solution $\phi:U\to \D$ of the
equation $(T-z)\phi_z=0$ $(z\in U)$ is the zero solution. This is equivalent to requiring that the analytic
local resolvent function $\phi$ in \eqref{eq:localresolvent} is determined uniquely, so that we can speak of ``the" analytic local resolvent function on $\rho_T(x)$. If $T$ is globally defined and bounded, then this property is also equivalent to 0 being the only element in $\D$ with empty local spectrum \cite[Proposition 1.2.16]{LaurNeu}. If $\D=X$, $T$ is bounded, and $T$ has the single-valued extension property, then
\begin{equation}\label{eq:localspectralradiusformulawithlimsup}
\limsup_{n\to\infty}\Vert T^n x\Vert^{1/n}=\sup\{|z| : z\in\sigma_T(x)\}\qquad(x\in X),
\end{equation}
cf.\ \cite[Proposition~3.3.13]{LaurNeu}.
Namely, the annulus of convergence of the Neumann-type series in \eqref{eq:Neumannseries} is $\{z : |z| > \limsup_{n\to \infty} \| T^n x\|^{1/n}\}$, and, since $T$ has the single-valued extension property, the sum of the series on this annulus necessarily equals the local resolvent on that open set. Hence from the standard theory of vector valued analytic functions it follows that this annulus must coincide with $\{z : |z| > \sup\{|z| : z\in\sigma_T(x)\} \}$, yielding \eqref{eq:localspectralradiusformulawithlimsup}. Naturally, the supremum in the right hand side of \eqref{eq:localspectralradiusformulawithlimsup} is a maximum if $T$ is bounded, but for unbounded operators this need not be the case. For obvious reasons, we will refer to the validity of \eqref{eq:localspectralradiusformulawithlimsup} in the extended positive real numbers as a local spectral radius formula for $T$ at $x$, in the general context where $X$ is a (complex) Banach space, $\D\subset X$ is a linear subspace, $T:\D\to X$ is a possibly unbounded and not necessarily closed operator with domain $\D$, and $x\in\D(T^\infty)$, and where the analytic local resolvent functions are assumed to take their values in $\D$.

More is known for globally defined bounded operators than just the validity of \eqref{eq:localspectralradiusformulawithlimsup}, for all $x\in X$, if $T$ has the single-valued extension property. If $T$ has this property, then, by \cite{PruPu, Vasipaper, Vrb}, the set of those $x$ in $X$ for
which $\sigma_T(x)\neq \sigma_T$ is of the first category
in $X$. By \cite[Proposition~3.3.14]{LaurNeu}, it is always
true, also in the absence of the single-valued extension property, that the set of $x\in X$ for
which $r_T(x)$ is equal to the spectral radius of $T$ is of the
second category in $X$. If $T$ has Bishop's property ($\beta$)
(see \cite[Definition~1.2.5]{LaurNeu} -- it is immediate that
property ($\beta$) implies the single-valued extension property), then, by
\cite[Proposition~3.3.17]{LaurNeu},
$r_T(x)=\lim_{n\to\infty}\Vert T^n x\Vert^{1/n}$, for all $x$ in
$X$, so that the local spectral radius formula
\eqref{eq:localspectralradiusformulawithlimsup} holds in a stronger form as
\begin{equation}\label{eq:localspectralradiusformulawithlim}
\lim_{n\to\infty}\Vert T^n x\Vert^{1/n}=\max\{|z| : z\in\sigma_T(x)\}\qquad(x\in X).
\end{equation}
Furthermore, a decomposable (see
\cite[Definition~1.1.1]{LaurNeu}) operator has property
($\beta$) by \cite[Theorem~1.2.7]{LaurNeu}, hence
\eqref{eq:localspectralradiusformulawithlim} holds for decomposable
operators.

For globally defined bounded operators, therefore, there exist general theorems asserting the validity of the local spectral radius formulas \eqref{eq:localspectralradiusformulawithlimsup} or \eqref{eq:localspectralradiusformulawithlim} for fairly general classes of operators. Much less seems to be known, however, for unbounded operators. Even for closed operators, we are not aware of results in that direction, even though there are dedicated monographs for local spectral theory for unbounded operators available, cf.\ \cite{ErdWang, Vasilescu}. The key tool in the case of globally defined bounded operators is the a priori knowledge that there exists, for all large enough $z$, an analytic local resolvent function given by \eqref{eq:Neumannseries}. This argument breaks down in the unbounded case and that may account for the (to our knowledge) absence of results asserting the validity of a local spectral radius formula for unbounded operators, in a generality comparable to that in the case of a globally defined bounded operator. As already mentioned, for closed operators, \cite[Proposion~4.9]{CJS} yields a stronger version of \eqref{eq:inequalityforclosedoperators}, but under a more restrictive hypothesis, and that is still only one of the two inequalities needed. Apart from that, only a few results related to (the two possible definitions of) the local spectral radius  for an unbounded operator at an element of its domain seem to be known, and these are of a more specialised nature. We will now discuss these.

First of all, if $T$ is a paranormal operator on domain $\D$ in a Hilbert space, i.e., an operator such that $\Vert Tx\Vert^2\leq\Vert T^2x\Vert \Vert x\Vert$, for all $x$ in $\D(T^2)$, then $\lim_{n\to\infty}\Vert T^n x\Vert^{1/n}$ exists in the extended positive real numbers; see \cite[Proposion~4.9]{CJS}. A possible relation with the local spectrum of $T$ at $x$ is still open to investigation, however.

The second example, again from \cite{CJS}, is for an algebraic operator $T$ on a domain $\D$ in a Hilbert space, i.e., an operator such that exists a non-zero polynomial $P\in\C[X]$ such that $p(T)x=0$ for all $x\in\D(T^{\textup{deg\,}P})$. In that case, one factors the unique monic minimal such polynomial as $(X-z_1)^{n_1}\cdots(X-z_m)^{n_m}$, and there is decomposition of $\D(T^\infty)$, which is assumed to be a non-zero subspace, as
\[
\D(T^\infty)=\bigoplus_{i=0}^m \textup{ker\,}(T_{\D(T^\infty)}-z_i)^{n_i},
\]
where $T_{\D(T^\infty)}$ denotes the restriction of $T$ to $\D(T^\infty)$. For $x$ in $\D(T^\infty)$, write $x=\sum_{i=0}^m x_i$, with $x_i\in\textup{ker\,}(T_{\D(T^\infty)}-z_i)^{n_i}$ $(i=1,\ldots,m)$. Then, by \cite[Proposition~6.2]{CJS},
\begin{equation}\label{eq:CJSformula}
\lim_{n\to\infty}\Vert T^n f\Vert^{1/n}=\max \{|z_j| : j=1,\ldots,m,\, x_j\neq 0\}.
\end{equation}
Although it is not investigated in \cite{CJS}, the analogy with the finite dimensional case suggests that the right hand side in \eqref{eq:CJSformula} is the local spectrum at $x$ of the restriction of $T$ to a suitable domain, in which case \eqref{eq:CJSformula} could be interpreted as a local spectral radius formula for that restriction.

The third example is taken from \cite{AdJ1}. If $P\in\C[X_1,\ldots,X_d]$ is a polynomial in $d$ variables, let $P(\partial)$ be the corresponding constant coefficient differential operator. Consider the Schwartz space $\S(\R^d)$ of rapidly decreasing functions as a subspace of $L^1(\R^d, dx)$. Then $P(\partial): \S(\R^d)\to L^1(\R^d)$ is closable, and we let $T$ be its closure. Naturally, $\S(\R^d)\subset\D(T^\infty)$. By \cite[Corollary~5.4]{AdJ1}, the unbounded operator $T$ on $L_1(\R^d)$ has the single-valued extension property, and, for $f\in\S(\R^d)$,
\begin{equation}\label{eq:localspectrumFourier}
\sigma_{T}(f)=\{P(i\lambda) : \lambda\in\supp \F f\}^\cl,
\end{equation}
where $A^{\cl}$ denotes the closure of a subset $A$ of the complex plane, and where the Fourier transform $\F f$ of $f$ is defined as
\[
\F f (\lambda) =\frac{1}{(2\pi)^{d/2}} \int _{\R^d} f(x) e^{-i \lambda\cdot x}\, dx \qquad (\lambda \in \R^d).
\]
Since, by \cite[Theorem~2.5]{AdJ1},
\begin{equation}\label{eq:limitexpressionFourier}
\lim_{n\to \infty}\| T^n f\| _1 ^{1/n} = \sup\left\{|z| : z\in\{P(i\lambda) : \lambda\in\supp \F f\}^\cl\right\}\qquad(f\in\S (\R^d)),
\end{equation}
we infer from the combination of \eqref{eq:localspectrumFourier} and \eqref{eq:limitexpressionFourier} the validity of the local spectral radius formula
\begin{equation}\label{eq:localspectralradiusformulaFouriercontinuous}
\lim_{n\to \infty}\| T^n f\| _1 ^{1/n} = \sup\{|z| : z\in\sigma_{T_1}(f)\}\qquad(f\in\S(\R^d))
\end{equation}
in the extended positive real numbers \cite[Corollary~5.4]{AdJ1}. Furthermore, although it was not stated as such in \cite{AdJ1}, \eqref{eq:localspectrumFourier} and \eqref{eq:localspectralradiusformulaFouriercontinuous} are, together with the assertion that $T$ has the single-valued extentions property, equally true, and with similar proofs, for $T=P(\partial)$ defined on the original domain $\D=\S(\R^d)$, viewed as a subspace of $L^p(\R^d)$, for all $1\leq p\leq\infty$. For $1<p\leq\infty$, the questions concerning the validity of \eqref{eq:localspectrumFourier} and the single-valued extension property of $T$, when one considers, as for $p=1$, a closed extension $T$ of $P(\partial)$, are open at the time of writing, cf.\ \cite[Conjecture~5.5]{AdJ1}.\footnote{The case $p=1$ is more manageable, because then the Fourier transform of an element of the domain of an arbitrary extension of $P(\partial)$ is a priori known to be a continuous function, cf.\ the proofs of Lemma~5.1 and 5.2 in \cite{AdJ1}.}
\\It should be noted here that the first result in the vein of \eqref{eq:limitexpressionFourier} was obtained (as late as 1990) by Bang \cite{Ba}, in one dimension and for the polynomial $P(t)=t$, and with proofs which are less elementary than those in \cite{AdJ1}. Several papers in this vein have appeared since, also for other ``diagonalising" transforms, with continuous or discrete spectral parameter. The corresponding analogues of \eqref{eq:limitexpressionFourier} are also known as ``real Paley-Wiener theorems", because they relate the growth rate of sequences such as $\{\Vert P(\partial)^n f\Vert_p\}_{n=1}^\infty$ to the support of $\F f$, just as for the complex Paley-Wiener theorems, but with everything taking place in the real domain. We refer to \cite[Section~4]{AdJ1} for an extensive overview of the literature on real Paley-Wiener theorems. The possible link with local spectral theory seems to have been worked out first in \cite[Section~5]{AdJ1}.

The fourth example is concerned with the local spectrum, at an element of $C^\infty(G)$, of a closed extension, in $L^p(G)$ ($1\leq p\leq\infty$), of an element of the center of the complexified universal enveloping algebra of a compact connected Lie group, with original domain $C^\infty(G)\subset L^p(G)$. In \cite{AdJ2}, using \cite{Sug}, the local spectrum of such an operator at a smooth function is determined, and, paralleling the result in \cite{AdJ1} for the Fourier transform on $\R^d$ as mentioned above, it is shown that the local spectral radius formula holds. In the case of the 1-torus and the operator $d/dt$, this formula was also obtained by Bang \cite{Ba}, but without its interpretation in terms of local spectral theory, and with different proofs. In Section~\ref{subsec:Fourier} we will return to this example. The key property that was used in \cite{AdJ2}, is that the (operator valued) Fourier transform on $G$ establishes a topological isomorphism between $C^\infty(G)$ and a space of rapidly decreasing (operator valued) sequences. We surmise, supported by the observation that results such as \cite[Theorem~0.3.4]{Helg} (for the $n$-sphere) indicate that this topological isomorphism property in \cite{Sug} may hold in greater generality, that the results in \cite{AdJ2} can be generalised, so that one can determine the local spectrum of (extensions of) invariant differential operators at $C^\infty$-functions on compact symmetric spaces, viewed as elements of $L^p$-spaces ($1\leq p\leq\infty)$ with respect to the invariant measure, and establish a local spectral radius formula for such operators at $C^\infty$-functions.

\subsection{Framework and results}\label{subsec:frameworkandresults}

We will now formulate the basic context and assumption, and possible additional assumptions, which allow us to obtain local spectral results for unbounded operators. These results are summarised, and subsequently we will comment on the relation of the general framework with concrete situations, indicating the range of applications.

\begin{setup}
Let $X$ be a complex Banach space, $\D\subset\D_e\subset X$ two not necessarily dense linear subspaces,
and $T: \D \to X$, $T_e:\D_e\to X$ two linear operators, not necessarily bounded or closed, with $T_e$ extending $T$.
\\ Let $\Lambda$ be a non-empty set and, for each $\lambda \in \Lambda$, let $L_\lambda$ be a normed space. If $s$ is an element of $\dirprod$, then, for $\lambda\in\Lambda$, we let $s(\lambda)$ denote the $\lambda^{\textup{th}}$-coordinate of $s$. We will assume in what follows that there exist a linear map $\F : X\to \dirprod$, injective on $\D_e$, and a map $\eig : \Lambda \to \C$, such that
\begin{equation}\label{eq:diagonalisingtransform}
\F(T_ex)(\lambda) = \eig(\lambda)\F x(\lambda),
\end{equation}
for all $\lambda \in \Lambda$, and all $x\in \D_e$.
\end{setup}

Of course, $T_e$ and $T$ can coincide. In the examples to be considered in Section~\ref{sec:examples}, $T$ is typically a differential operator on a space $\D$ of smooth functions, and $T_e$ is a suitable extension of $T$.

Under the basic assumption, we will investigate whether $T_e$ as above has the single-valued extension property, determine the local spectrum of $T_e$ at an element of $\D$ (not $\D_e)$ and
establish a local spectral radius formula for $T_e$ at an element of $\D(T^\infty)$ (not $\D(T_e^\infty))$ under (a combination of) suitable additional assumptions, which we now formulate.

\begin{assumptions}\qquad
\begin{enumerate}
\item[(A1)] Assume that, for each $\lambda \in \Lambda$, the map $x\mapsto \F x(\lambda)$ is continuous on $\D_e$ in the topology induced by $X$.

\item[(A2)] Assume that there exists a family $G_\lambda$ $(\lambda\in\Lambda)$ of not necessarily bounded linear maps $G_\lambda : L_\lambda \to \D$,
such that, for all $x\in \D$, the series
\[
\sum _{\lambda \in \Lambda} G_\lambda (\F x(\lambda)),
\]
converges absolutely in $X$, with sum equal to $x$.

\item[(A3)]  Assume that $\D$ has a topology of its own, such that the inclusion map $\D\hookrightarrow X$ is continuous.
    \\ Furthermore, let $\familyoffunctions$ be a non-empty family of functions $\seminormfunction: \Lambda\to[0,\infty)$, and let $\S(\Lambda)$ be the subspace of those $s\in\dirprod$ such that $\sup _{\lambda \in \Lambda} \seminormfunction(\lambda)\|s(\lambda)\|<\infty$, for all $\seminormfunction\in\familyoffunctions$, equipped with the locally convex topology determined by the seminorms $q_\seminormfunction$ $(\seminormfunction\in\familyoffunctions)$, where $q_\seminormfunction(s) = \sup _{\lambda \in \Lambda} \seminormfunction(\lambda)\|s(\lambda)\|$ $(s\in\S(\Lambda))$. Furthermore, assume that $\F(\D)=\S(\Lambda)$ and that $\F:\D \to \S(\Lambda)$, where $\D$ carries its own topology, is an isomorphism of topological vector spaces.
\end{enumerate}
\end{assumptions}

Our results can then be summarised as follows.

\begin{summaryofresults}Under the basic assumption:
\begin{enumerate}
\item[(1)] (A1) implies the single-valued extension property for $T_e$ (Theorem~\ref{thm:A1SVEP});
\item[(2)] (A1) and (A2) together imply that $\lim_{n\to\infty}\Vert T_e^n x\Vert^{1/n}$ exist, for $x\in\D(T^\infty)$, and can be related to the set which will occur as a local spectrum in (3)  (Theorem~\ref{thm:localspectralradiusexpression});
\item[(3)] (A3) affords a description of $\sigma_{T_e}(x)$, for $x\in\D$ (Theorem~\ref{thm:localspectrum});
\item[(4)] (A1), (A2), and (A3) together imply, as a direct consequence of (2) and (3), the validity of a local spectral radius formula for $T_e$ at $x\in\D(T^\infty)$ (Theorem~\ref{thm:localspectralradiusformula}).
\end{enumerate}
\end{summaryofresults}

Let us discuss the general framework in the basic and possible additional assumptions, and relate it with more concrete situations.

Obviously, the basic assumption stipulates that $\D_e$ is, as an abstract vector space, isomorphic, via the transform $\F$, with a subspace of $\dirprod$. In practical situations, $\F$ will be injective on the whole of $X$, but for the proofs injectivity on $\D_e$ is sufficient. This abstract embedding of $\D_e$ diagonalises $T_e$, with $\eig$ describing the eigenvalues. In the examples in the present paper, $\Lambda$ will be a subset of a lattice in $\R^d$, and the $L_\lambda$ will all be equal to $\mathbb C$. An element $s$ of $\dirprod$ can then be identified with a complex-valued sequence. The formalism and the proofs support more general situations, however, so we need not restrict our attention to this particular situation. In fact, the results in \cite{AdJ2} would not be covered by such more restrictive hypotheses, whereas they do fall within the scope of our more general formalism; see also Remark~\ref{rem:compactsymmetricspaces} in Section~\ref{subsec:Fourier}.

The possible additional assumptions (A1), (A2) and (A3) are all related to the properties of $\F$ and the image of $\D$. A typical example would be the following, where we first concentrate on the case where $T_e=T$. Suppose that $\D$ is a subspace of a separable complex Hilbert space $H$ (which we take for the Banach space $X$), with inner product $(\,.\,,\,.\,)$, and that $\{e_n\}_{n=0}^\infty$ is an orthonormal basis of $H$ contained in $\D$. Hence $\D$ is dense in $H$. We suppose, furthermore, that $T$ is (for simplicity) symmetric (which is meaningful because $\D$ is now known to be dense), and that, for $n\in\N_0$, there exists $\eig(n)\in\C$ such that $Te_n=\eig(n)e_n$. Then $\eig:\N_0\to\C$ is real-valued. Let $\Lambda=\N_0$, and put $L_n=\C$, for $n\in\N_0$. Define $\F:H\to\prod_{n\in\N_0}\C$ by $\F x(n)=(x,e_n)\,\,(x\in H,\,n\in\N_0)$, and, for $n\in\N_0$, define $G_n:\C\to \D$ by $G_n(z)=ze_n$ ($z\in \C=L_n$). Then, for $x\in\D$, \eqref{eq:diagonalisingtransform} holds, as we have already observed that $\eig$ is real-valued. Hence the basic diagonalising assumption is satisfied for $T$, and $\F$ is even injective on the whole of $H$. Certainly (A1) holds, and, furthermore, for all $x\in\D$, we have
\begin{equation}\label{eq:HilbertFourierseries}
x=\sum_{n=0}^\infty (x,e_n)e_n=\sum_{n=0}^\infty G_n(\F x(n)),
\end{equation}
where the series is norm convergent in $H$. It is, however, not automatic that the series is absolutely convergent in $H$, as required in (A2): this is equivalent with the sequence $\{(x,e_n)\}_{n=0}^\infty$ being in the proper subspace $\ell^1(\N_0)$ of $\ell^2(\N_0)$. Continuing with (A3), for $k=0,1,2,\ldots$, let $\seminormfunction_k(n)=(1+n)^k\,\,(n\in\N_0)$, and let $\familyoffunctions$ consist of the functions $\seminormfunction_k:\N_0\to[0,\infty)$ ($k=0,1,2,\ldots$). Then $\S(\N_0)$ is the space of rapidly decreasing sequences, supplied with its usual Fr\'echet topology. As we had already observed, it is not guaranteed that $\F$ maps $\D$ into $\ell^1(\N_0)$, and much less so that it maps $\D$ into $\S(\N_0)$. However, this latter property \emph{does} hold in a number of well known cases, where, in fact, $\F$ then even establishes the topological isomorphism between $\D$ and $\S(\N_0)$ as required in (A3). Certainly (A2) will then also be satisfied, since the fact that $\Vert(x,e_n)e_n\Vert=|(x,e_n)|$ $(n\in\N_0)$ implies that the series in \eqref{eq:HilbertFourierseries} is absolutely convergent if the sequence $\{(x,e_n)\}_{n=0}^\infty$ is rapidly decreasing.

The cases, where $\F$ as defined above is a topological isomorphism between $\D$ and $\S(\N_0)$, that we have in mind, are those where $\D$ is, for example, the space of $C^\infty$ functions on a sufficiently regular subset of $\R^d$, for some $d$, viewed as a subspace of a suitable $L^2$-space $H$, and where $T$ is a differential operator, defined on $\D$ and leaving $\D$ invariant (so that $\D=\D(T^\infty)$). If $T$ is symmetric on $H$, and if there exists an orthonormal basis $\{e_n\}_{n=0}^\infty\subset\D$ of $H$ consisting of eigenfunctions of $T$, where the eigenvalue $\eig(n)$ in $Te_n=\eig(n) e_n$ grows at least polynomially with $n$, then, for $x\in\D$, the sequence $\{(x,e_n)\}_{n=0}^\infty$ is in $\S(\N_0)$, by the standard argument involving the symmetry of $T$ and the fact that the Fourier coefficients with respect to an orthonormal basis tend to zero. This rapid decay can be thought of as a manifestation of the informal general principle that ``smoothness often gives good convergence''. Certainly, the absolute convergence in (A2) will then hold. However, in a number of cases, this principle manifests itself in an even better form: when $\D$ is supplied with a topology appropriate for a space of test functions, $\F$ turns out, in those cases, to be a topological isomorphism between $\D$ and $\S(\N_0)$, and the series in \eqref{eq:HilbertFourierseries} even converges in the own topology of $\D$. As the continuity of the inclusion of $\D$ into $H$ is usually an innocent one, (A3) will then be satisfied.

Once it has thus been established that the basic assumption, and the assumptions (A1), (A2), and (A3), are satisfied with $T=T_e$, it is in this class of practical examples usually easy to find a proper extension $T_e$ of $T$ on a domain $\D_e$, such that all assumptions, and in particular \eqref{eq:diagonalisingtransform}, are satisfied.

Actually, one can in these situations often go further and pass from an $L^2$ Hilbert space as above to $L^p$-spaces. We will now discuss this, and again we first concentrate on the case where $T_e=T$. The important point is that the requirement in (A3) that $\F$ is a topological isomorphism does not involve the Hilbert space. This space in which $\D$ can be viewed ``originally" may be a means, or a guide, to obtain $\F$, but once the required properties have been established, there is no need to retain it. One can often consider $\D$, in its own topology, as a subspace of an $L^p$-space, and regard $T$ accordingly. All but one of the remaining hypotheses are then easily verified or falsified. To start with, it is often possible to view such an $L^p$-space as a subspace of the continuous dual $\D^\prime$ of $\D$, i.e., as a space of distributions, and since $\F$ usually has a natural injective extension from $\D$ to $\D^\prime$, $\F$ can be defined injectively on such $L^p$-spaces as well. The verification that this extension satisfies the basic assumption again will then be a mere formality. Furthermore, the question whether the inclusion of $\D$ into an $L^p$-space is continuous, is easily answered in concrete situation, as is the question concerning continuity in (A1). There is only one matter that needs serious attention, and that is the absolute convergence of the series in \eqref{eq:HilbertFourierseries}. As already observed above, this series is evidently absolutely convergent in an $L^2$-context if the coefficients $\{(x,e_n)\}_{n=0}^\infty$ form a rapidly decreasing sequence, since the norm of $e_n$ is then equal to 1, for all $n\in\N_0$. When viewing the $e_n$ as elements of $L^p$-spaces, an extra argument is needed at this point. In the examples in the present paper, this will be provided by known estimates on the special functions involved. These will show that the sequence formed by the norms of the $e_n$, in each of the $L^p$-spaces then under consideration, is slowly increasing. Since, for $x\in\D$, $\{(x,e_n)\}_{n=0}^\infty$ is already known to be a rapidly decreasing sequence, this shows that, for $x\in\D$, the series in \eqref{eq:HilbertFourierseries} is absolutely convergent in each of the $L^p$-spaces then under consideration, as required in (A2).

Having arrived at this point, our results will then apply with $T_e=T$, viewed as an operator on $\D$ in an $L^p$-space under consideration. As in the Hilbert space context discussed previously, suitable extensions of $T$ in these $L^p$-spaces to which these results also apply are then usually easily determined. We refer to the examples in Section~\ref{sec:examples} for illustrations.

As an informal summary of the above discussion: if $T$ is a differential operator on a space $\D$ of test functions, which is (for example) symmetric with respect to a suitable inner product, and which can be diagonalised such that the corresponding Fourier-type transform establishes a topological isomorphism between $\D$ and $\S(\N_0)$ (i.e., when ``very good convergence results for eigenfunction expansions of smooth functions hold"), then the basic assumption and the assumptions (A1), (A2), and (A3) can be expected to be satisfied with $T=T_e$, viewed as an operator with domain $\D$ in suitable $L^p$-spaces, and the results in local spectral theory as summarised above can expected to be valid. Moreover, it is to be expected that these results will also apply to suitable extensions of $T$ in those $L^p$-spaces.

\section{Proofs}\label{sec:resultsandproofs}

We will now investigate the local spectral properties of $T_e$ under (a combination of) the hypotheses (A1), (A2), and (A3), stating and proving the precise versions of the results already summarised in Section\ref{subsec:frameworkandresults}.

The support $\supp s$ of an element $s\in\dirprod$, which will be needed in the description of local spectra, is naturally defined as
\[
\supp s=\{\lambda\in\Lambda : s(\lambda)\neq 0\}.
\]
We will adhere to the convention that the supremum of an empty set of real numbers is 0. If $A$ is a subset of the complex plane, then $A^{\cl}$ denotes its closure.

\begin{theorem}\label{thm:A1SVEP}
Suppose that the basic assumption and (A1) hold. Then the operator $T_e$ has the single-valued extension property.
\end{theorem}
\begin{proof}
Suppose that $U\subset\C$ is open and non-empty, that $\phi: U\to\D_e$ is analytic in the topology of $\D_e$ induced by $X$, and that $(T_e-z)\phi_z=0$, for all
$z\in U$. We must show that $\phi=0$.
Applying $\F$, we see that $(\eig(\lambda)-z)\F \phi_z(\lambda)=0$,
for all $z\in U$, and all $\lambda\in\Lambda$. If $\lambda\in\Lambda$ is fixed, we conclude from this that
$\F \phi_z(\lambda)=0$, for all $z\in U$, with at most one exception.
However, the continuity of $\phi$ and (A1) imply that $z\mapsto \F\phi_z(\lambda)$
is continuous on $U$, so there cannot be an exceptional point. We thus have $\F \phi_z(\lambda)=0$, for all $z\in U$. From the injectivity of $\F$ we then conclude that $\phi_z=0$ for all $z\in U$, as required.

\end{proof}

\begin{proposition}\label{prop:liminfinequality}
Suppose that the basic assumption and (A1) hold. If $x\in \D(T_e^\infty)$, then
\begin{equation}\label{eq:liminfestimate}
\sup\left \{|\eig(\lambda)| : \lambda\in\supp \F x\right\} \le \liminf _{n\to \infty} \|T_e^n x\|^{1/n},
\end{equation}
in the extended positive real numbers.
\end{proposition}
\begin{proof}
Let $x\in\D(T_e^\infty)$. We may assume that the right hand side in \eqref{eq:liminfestimate} is finite. For $n=1,2,\ldots$, we have
\[
|\eig(\lambda)|^n \|\F x(\lambda)\| = \|(\F(T_e^n x))(\lambda)\|\le \|\F_\lambda \| \|T_e^n x\|,
\]
where $ \|\F_\lambda \|$ is the norm of the map $x\mapsto \F x(\lambda)$, from $\D_e$ to $L_\lambda$, which is finite by assumption (A1).
If $\lambda\in\supp \F x$, then this implies
\[
|\eig(\lambda)|^n \le \frac{\|\F_\lambda \| \|(T_e^n x)(\lambda)\|}{\|\F x(\lambda)\|}\qquad(n=1,2,3\ldots).
\]
From this we see that $|\eig(\lambda)|\leq \liminf _{n\to \infty} \|T_e^n x\|^{1/n}$, both when $\Vert\F_\lambda\Vert\neq 0$ and when $\Vert\F_\lambda\Vert=0$. Hence the result follows.
\end{proof}

\begin{proposition}\label{prop:limsupinequality}
Suppose that the basic assumption and (A2) hold. If $x\in\D(T^\infty)$, then
\begin{equation}\label{eq:limsupestimate}
\limsup _{n\to \infty} \|T^n x\|^{1/n} \le \sup\left \{|\eig(\lambda)| : \lambda\in\supp \F x\right\} ,
\end{equation}
in the extended positive real numbers.
\end{proposition}
\begin{proof}
Let $x\in\D(T^\infty)$. By assumption (A2), we have
\begin{equation*}
T^n x = \sum _{\lambda \in \Lambda} G_\lambda (\F(T^n x)(\lambda)) =
\sum _{\lambda \in \Lambda} \eig(\lambda)^n G_\lambda (\F x(\lambda))\qquad(n=1,2,\ldots),
\end{equation*}
hence
\begin{align}\label{eq:realseries}
\|T^n x\| &\le \sup\left \{|\eig(\lambda)|^n : \lambda\in\supp \F x\right\}
\sum _{\lambda \in \Lambda} \| G_\lambda (\F x(\lambda))\|
\\ &= \left[\sup\left \{|\eig(\lambda)| : \lambda\in\supp \F x\right\}\right]^n
\sum _{\lambda \in \Lambda} \| G_\lambda (\F x(\lambda))\|\qquad(n=1,2,\ldots),\notag
\end{align}
where the series is convergent in the positive real numbers by assumption (A2). We may assume that the right hand side in \eqref{eq:limsupestimate} is finite. Hence \eqref{eq:realseries} implies the statement in the Proposition for all $x$ in $\D(T^\infty)$, such that $\sum _{\lambda \in \Lambda} \| G_\lambda (\F x(\lambda))\|\neq 0$. The remaining case where $\sum _{\lambda \in \Lambda} \| G_\lambda (\F x(\lambda))\|=0$ is equivalent with $x=0$, by assumption (A2), and then the statement in the Proposition holds by convention.
\end{proof}

Combining the Propositions~\ref{prop:liminfinequality} and \ref{prop:limsupinequality} with the observation that, for $A\subset\C$, $\sup\{|z| : z\in A\}=\sup\{|z| : z\in A^\cl\}$ in the extended positive real numbers, we have the following result.

\begin{theorem}\label{thm:localspectralradiusexpression}
Suppose that the basic assumption, (A1) and (A2) hold. If $x\in \D(T^\infty)$, then
\[
\lim _{n\to \infty} \|T_e^n x\|^{1/n} = \sup\left \{|z| : z\in\left[\eig(\supp \F x)\right]^\cl\right\}
\]
in the extended positive real numbers.
\end{theorem}

Next we turn to the description of a local spectrum.

\begin{theorem}\label{thm:localspectrum}
Suppose that the basic assumption and (A3) hold. If $x\in\D$, then
\[
\sigma_{T_e}(x)= \left[\eig(\supp \F x)\right]^\cl.
\]
\end{theorem}

\begin{proof}
We first establish that
\begin{equation}\label{eq:inclusion}
\left\{\eig(\lambda) : \lambda\in\supp \F x\right\}^\cl\subset\sigma_{T_e}(x).
\end{equation}
Suppose that $\lambda _0 \in \Lambda$ is such that $\eig(\lambda_0)\in \rho_{T_e}(x)$.
Then there exist a neighborhood $U$ of $\eig(\lambda_0)$, and an analytic function $\phi:U\to\D_e$, such that $(T_e - z)\phi_z=x$, for all $z\in U$. Applying the map $\F$, we find that
\[
(\eig(\lambda)-z)(\F\phi_z)(\lambda)=\F x(\lambda) \qquad (z\in U,\,\lambda\in\Lambda).
\]
Since $\eig(\lambda_0)$ is in $U$, we can choose $z=\eig(\lambda_0)$, and on simultaneously specialising to $\lambda=\lambda_0$, we find that $\F x(\lambda_0)=0$.
Hence
\[
\left\{\eig(\lambda) : \lambda\in\supp \F x\right\} \subset \sigma_{T_e}(x).
\]
Since the right hand side is closed, \eqref{eq:inclusion} follows.

Next, we show the reverse inclusion of \eqref{eq:inclusion}, which is equivalent to
\begin{equation}\label{eq:reverseinclusion}
\C\setminus\left\{\eig(\lambda) : \lambda\in\supp \F x\right\}^\cl\subset\rho_{T_e}(x).
\end{equation}
Suppose $z_0\notin \left\{\eig(\lambda) : \lambda\in\supp \F x\right\}^\cl$,
and let $\epsilon >0$ be such that $|\eig(\lambda) - z_0| > \epsilon$, for all $\lambda\in\supp \F x$.
Let $U = \{ z\in \C : |z-z_0|< \frac{\epsilon}{2}\}$. We will construct an analytic function $\phi: U\to\D$, such that
\[
(T-z)\phi_z=x\qquad(z\in U).
\]
Since $\D\subset\D_e$, and $T_e$ extends $T$, this implies that $z_0\in \rho_{T_e}(x)$, so that the proof of \eqref{eq:reverseinclusion}, and consequently that of the Theorem, will then be complete.

Hence it remains to construct $\phi$.  This is done by first solving the problem in the transformed picture. To start with, note that, for $z\in U$, and $\lambda\in\supp \F x$, one has $|\eig(\lambda) - z| > \frac{\epsilon}{2}$. Hence, for $z\in U$, we can define $\psi_z \in \dirprod$ as
\begin{equation*}
\psi_z(\lambda)=
\begin{cases}
\displaystyle{\frac{\F x(\lambda)}{\eig(\lambda)-z}}&\textup{if }\lambda\in\supp \F x;\\
0 &\textup{if }\lambda\notin\supp \F x.
\end{cases}
\end{equation*}
For $z\in U$ fixed, note that $\| \psi_z (\lambda)\|\le \frac{2}{\epsilon} \| \F x (\lambda)\|$, if $\lambda\in\supp \F x$,
and that this is equally true if $\lambda\notin\supp \F x$. Since $\F x\in\S(\Lambda)$, this implies that $\psi_z \in \S(\Lambda)$.
Hence we have a map $\psi:U\to\S(\Lambda)$, defined by $z\mapsto\psi_z$. We claim that $\psi$ is analytic. In establishing this claim, we will show that its derivative is the map $\chi:U\to\S(\Lambda)$, given, for $z\in U$, by
\begin{equation*}
\chi_z(\lambda)=
\begin{cases}
\displaystyle{\frac{\F x(\lambda)}{(\eig(\lambda)-z)^2}}&\textup{if }\lambda\in\supp \F x;\\
0 &\textup{if }\lambda\notin\supp \F x.
\end{cases}
\end{equation*}
Just as for $\psi_z$, $\chi_z$ is actually in $\S(\Lambda)$, for $z\in U$, and not just in $\dirprod$.

We now show that $\psi$ is analytic at an arbitrary fixed point $z_1$ in $U$. If $h\in\C$ and $|h|<\frac{\epsilon}{2}-|z_1-z_0|$, then $z_1+h\in U$. For such $h$ with $h\neq 0$ one has, for $\lambda\in\supp \F x$,
\[
\frac{\psi_{z_1+h}(\lambda)-\psi_{z_1}(\lambda)}{h} - \chi_{z_1}(\lambda)=
\frac{h\F x(\lambda))}{(\eig(\lambda)-z_1)^2(\eig(\lambda)-z_1-h)}.
\]
If $|h|<\frac {\epsilon}{4}$, then, for $\lambda\in\supp \F x$, we have
$|\eig(\lambda)-z_1-h|\ge |\eig(\lambda)-z_1|-|h|>\frac{\epsilon}{2}-\frac{\epsilon}{4}=\frac{\epsilon}{4}$,
so that, for $0<h\in\C$ with $|h|<\min(\frac {\epsilon}{4}, \frac{\epsilon}{2}-|z_1-z_0|)$,
\[
\left \|\frac{\psi_{z_1+h}(\lambda)-\psi_{z_1}(\lambda)}{h} - \chi_{z_1}(\lambda)\right \|
\le \left (\frac 2\epsilon\right )^2\frac 4\epsilon |h|\,\| \F x (\lambda)\|
= \frac {16}{\epsilon^3}|h|\| \F x (\lambda)\|.
\]
As this is trivially true for $\lambda\notin\supp \F x$, we conclude that, for all such $h$,
\[
\sup _{\lambda \in \Lambda} \seminormfunction(\lambda)\left \|\frac{\psi_{z_1+h}(\lambda)-\psi_{z_1}(\lambda)}{h} - \chi_z(\lambda)\right \|
\le \frac {16}{\epsilon^3}|h|\sup _{\lambda \in \Lambda} \seminormfunction(\lambda)
\| \F x (\lambda)\|,
\]
for all $\seminormfunction$ in $\familyoffunctions$. It is immediate from this that $\psi$ is analytic at $z_1$, as claimed, and that its derivative at $z_1$ is $\chi_{z_1}$, as announced.

After these preparations, the actual local resolvent $\phi:U\to\D$ is easily constructed. Indeed, let $\phi_{z}=\F^{-1}\psi_{z}$, for $z\in U$. Then $\phi$ is analytic on $U$ when $\D$ carries the topology from $X$, as a consequence of the above and both parts of (A3). Furthermore, for $z\in U$, and $\lambda \in \Lambda$,
\[
\F((T-z)\phi_z)(\lambda)=(\eig(\lambda)-z)(\F\phi_z)(\lambda)=(\eig(\lambda)-z)\psi_z(\lambda)=\F x(\lambda).
\]
Hence $(T-z)\phi_z=x$, for all $z\in U$, as needed.
\end{proof}

Finally, we can state the local spectral radius formula for the operator $T$.

\begin{theorem}\label{thm:localspectralradiusformula}
Suppose that the basic assumption, (A1), (A2), and (A3) hold. If $x\in \D(T^\infty)$, then the local spectral radius formula
\[
\lim _{n\to \infty} \|T_e^n x\|^{1/n} = \sup\left \{|z| : z\in \sigma_{T_e}(x)\right\}
\]
holds in the extended positive real numbers.
\end{theorem}

\begin{remark}Under the hypotheses of Theorem~\ref{thm:localspectralradiusformula}, the Theorems~\ref{thm:A1SVEP} and \ref{thm:localspectrum} also hold. Hence $T_e$ has the single-valued extension property, and the local spectrum $\sigma_{T_e}(x)$ of $T_e$ at $x\in\D(T^\infty)\subset\D$ is known.
\end{remark}

\begin{remark}
The validity of a local spectral radius formula as asserted in Theorem~\ref{thm:localspectralradiusformula} is, one could say, established ``by inspection". Indeed, for the diagonalisable operator operator under consideration, one can explicitly determine what the local spectrum at a suitable point is, hence what the corresponding local spectral radius (in its natural definition) is, and one can also determine $\limsup _{n\to \infty} \|T^n x\|^{1/n}$ (which, incidentally, is the corresponding limit). Then one simply observes that these two extended positive real numbers are equal. This is a difference with the demonstration, in Section~\ref{subsec:introductionandliterature}, of the validity of a local spectral radius formula for a globally defined bounded operator with the single-valued extension property, where the local spectrum itself did not appear.
\end{remark}

\section{Examples}\label{sec:examples}

In this Section, we collect a number of concrete applications of the results in Section~\ref{sec:resultsandproofs}, along the lines as discussed in Section~\ref{subsec:frameworkandresults}. Based on the informal general principle on the relation between smoothness and convergence, as mentioned in that discussion, we trust that there are more, even though they have perhaps not yet been identified as such (see also Section~\ref{subsec:classofexamples}).

\subsection{Constant coefficient differential operators on $\T$}\label{subsec:Fourier}
Let $\T$ be the 1-torus, with total rotation invariant measure equal to 1, and let $\D = C^\infty (\T)$. As usual, we will identify functions on $\T$ with $2\pi$-periodic functions on $\R$. Reviewing basic notions from Fourier analysis and distribution theory, we introduce a Fr\'echet topology on $\D$ by means of the seminorms $p_k$, for $k\in\N_0$, where, for $g\in\D$, $p_k(g)=\max_{t\in[0,2\pi]}|g^{(k)}(t)|$, and we let $\Ddual$ denote the continuous dual of $\D$, i.e., $\Ddual$ is the space of distributions on $\T$. For $\dist\in\Ddual$, and $g \in\D$, we write $\langle\dist, g\rangle$ for the canonical pairing. Then $\D$ can be identified with a subspace of $\Ddual$ when, via the usual pairing
\[
\langle f,g\rangle=\frac{1}{2\pi}\int_0^{2\pi} f(t)g(t)\,dt\qquad(g\in\D),
\]
we view $f\in\D$ as an element of $\Ddual$.

Let $P\in\C[X]$ be a polynomial in one variable, and let
\[
T=P(d/dt)
\]
be the corresponding constant coefficient differential operator, with domain $\D$. Obviously, $\D$ is invariant under $T$, and $T:\D\to\D$ is continuous. Since
\begin{equation}\label{eq:symmetryFourier}
\langle T f,g \rangle=\langle f,P(-d/dt)g\rangle\qquad(f,g\in\D),
\end{equation}
we can extend $T$ from the embedded copy of $\D$ in $\Ddual$ to $\Ddual$ by
\begin{equation}\label{eq:extensiontodistributionsFourier}
\langle T\dist,g \rangle=\langle \dist,P(-d/dt)g\rangle\qquad(\dist\in\Ddual,\,g\in\D).
\end{equation}

We will employ the usual Fourier transform as a diagonalising transform for $T$ on $\D$. In fact, since it also diagonalises $T$ on $\Ddual$ it will, as a consequence, diagonalise the extensions of $T$ to be defined below in particular. For $n\in\Z$, let
\[
e_n(t)=e^{int}\qquad(t\in[0,2\pi]),
\]
so that
\begin{equation}\label{eq:eigenvalueFourier}
Te_n=P(in)e_n\qquad(n\in\Z).
\end{equation}

Let $\S(\Z)$ be the space of rapidly decreasing sequences on $\Z$, in its usual Fr\'echet topology, and define $\F:\D\to \S(\Z)$ by
\begin{align}\label{eq:transformfunctionsFourier}
\F(f)(n)&=\langle f,\overline{e}_n\rangle\\
&=\frac{1}{2\pi}\int_0^{2\pi} f(t)e^{-int}\,dt \qquad (f\in \D,\, n\in\Z).\notag
\end{align}
Since $\D\subset L^2(\T)$, and $\{e_n\}_{n\in\Z}$ is an orthonormal basis of $L^2(\T)$, a standard argument, using \eqref{eq:symmetryFourier} for $P=X$ and \eqref{eq:eigenvalueFourier}, implies that the image of an element of $\D$ under $\F$ is, in fact, an element of $\S(\Z)$, as was implicitly used in the notation above. However, as is well known, cf.\ \cite[Theorem~51.3]{Tre}, $\F$ is actually even an isomorphism of topological vector spaces between $\D$ and $\S(\Z)$, and, for $f\in\D$, the series
\begin{equation}\label{eq:seriesFourier}
\sum_{n\in\mathbb Z}\F(f)(n)e_n
\end{equation}
does not just converge to $f$ in $L^2(\T)$, as is obvious, but it also converges to $f$ in the topology of $\D$.

We can extend $\F$ from $\D$ to $\Ddual$, obtaining an injective map $\F:\Ddual\to\prod_{n\in\Z}\C$, by defining
\[
\F(\dist)(n)=\langle \dist,\overline{e}_n\rangle \qquad (\dist\in \Ddual,\, n\in\Z).
\]

Naturally, the image of $\F$ consists precisely of the slowly increasing sequences on $\Z$ (the continuous dual of $\S(\Z)$), but we will not need this. The relevant property is that $\F$ diagonalises $T$ on $\Ddual$, by the usual argument. Indeed, for $\dist\in\Ddual$, and $n\in\Z$,
\begin{align}\label{eq:diagonalisationFourier}
\F (T \dist)(n)&=\langle T \dist,\overline{e}_n\rangle
\\&=\langle \dist, P(-d/dt)\overline{e}_n\rangle\notag
\\&=\langle \dist, P(in)\overline{e}_n\rangle\notag
\\&=P(in)(\F \dist)(n).\notag
\end{align}

We now come to the application of the results in Section~\ref{sec:resultsandproofs}. Let $X= L^p(\T)$, for some fixed $1\leq p \leq \infty$. We view $\D$ as a subspace of $X$, so that $T$ is an operator on $X$ with domain $\D$. Furthermore, we embed $X$ as an abstract vector space into $\Ddual$ via the pairing
\[
\langle f,g\rangle=\frac{1}{2\pi}\int_0^{2\pi} f(t)g(t)\,dt\qquad(f\in X,\,g\in\D).
\]
Let $T_c$ be the operator on $X$ with domain $\D_c$, consisting of those $f\in X$ such that the distribution $Tf$, as defined in \eqref{eq:extensiontodistributionsFourier}, is in $X$, and defined, for such $f$, by $T_c f=Tf$. Then $ T_c$ is an extension of $T$, which, incidentally, is easily seen to be closed.

We will now proceed to show that the results in Section~\ref{sec:resultsandproofs} apply to each operator $T_e$ on $X$, such that
\begin{equation}\label{eq:extensionFourier}
T\subset T_e\subset T_c.
\end{equation}
To start with, as a consequence of \eqref{eq:diagonalisationFourier}, the injective map $\F$ on $\D_e$ then certainly diagonalises $T_e$ on $\D_e\subset\Ddual$, as required in the basic assumption \eqref{eq:diagonalisingtransform}, the eigenvalues being given by $\eig(n)=P(in)$ ($n\in\Z$). Hence the basic assumption is satisfied. Assumption (A1) is satisfied since $e_n\in L^q(\T)$ $(n\in\Z)$, where $q$ is the conjugate exponent of $p$. Assumption (A2) is also valid, since, for $f\in\D$, the series in \eqref{eq:seriesFourier} is absolutely convergent in $X$. Indeed,
\[
\sum_{n\in\mathbb Z}\Vert \F(f)(n)e_n\Vert_p=\sum_{n\in\Z}|\F f(n)|<\infty,
\]
as $\F f\in\S(\Z)$. Moreover, since we know already that the series in \eqref{eq:seriesFourier} converges to $f$ in the topology of $\D$, it also converges to $f$ in $X$, by continuity of the inclusion of $\D$ into $X$ (which, in the passing, is the first part of assumption (A3)). Hence assumption (A2) is satisfied. As already observed above, $\F$ is a topological isomorphism between $\D$ and $\S(\Z)$, which is the remaining part of (A3). We thus see that the basic assumption and (A1), (A2), and (A3) in Section~\ref{subsec:frameworkandresults} are all satisfied. Consequently, the Theorems~\ref{thm:A1SVEP}, \ref{thm:localspectralradiusexpression}, \ref{thm:localspectrum}, and \ref{thm:localspectralradiusformula} hold for each operator $T_e$ as in \eqref{eq:extensionFourier}. To summarise, we have the following.

\begin{theorem}\label{thm:summarizingtheoremFourier}
Let $1\leq p\leq\infty$. Let $\D=C^\infty(\T)$, let $P$ be a polynomial in one variable, and let $T=P(d/dt)$ be the corresponding constant coefficient differential operator, viewed as an operator on $L^p(\T)$ with domain $\D$. Let $T_e$ be an extension of $T$ on $L^p(\T)$ as in \eqref{eq:extensionFourier}. Then $T_e$ has the single-valued extension property.

Let $f\in\D$, and let $\F f:\Z\to\C$ be the Fourier transform of $f$ as in \eqref{eq:transformfunctionsFourier}. Then
\begin{equation}\label{eq:localspectralradiusformulaFourier}
\lim_{n\to\infty}\Vert T_e^n f\Vert_p^{1/n}=\sup\left\{|z| : z\in \{P(in) : n\in\Z,\, \F f(n)\neq 0\}^\cl\right\}
\end{equation}
in the extended positive real numbers. Moreover,
\[
\sigma_{T_e}(f)=\{P(in) : n\in\Z,\, \F f(n)\neq 0\}^\cl,
\]
so that \eqref{eq:localspectralradiusformulaFourier} is a local spectral radius formula for $T_e$ at $f$.
\end{theorem}

\begin{remark}\label{rem:compactsymmetricspaces} As a particular case of the above, we see that, for $1\le p\le \infty$:
\[
\lim _{n\to \infty} \left \|\frac{d^n f}{dt^n}\right \|_p^{1/n} = \sup\left \{|\F f(k)| : k\in\supp \F f\right\} ,
\]
for $f\in C^\infty (\T)$. This is \cite[Theorem~2]{Ba}, where the result was established using Paley-Wiener theory and the Bernstein inequality.
\end{remark}
\begin{remark}
Theorem~\ref{thm:summarizingtheoremFourier} has an obvious generalisation to the $d$-torus for $d\geq 2$, in which case the set-theoretical closure in the statement is not superfluous, as it is for $d=1$. In fact, it can be generalised to all compact connected Lie groups. This general case was treated in detail in \cite{AdJ2}. As already mentioned when discussing this generalisation as the fourth example in Section~\ref{subsec:introductionandliterature}, we surmise that similar results hold in still greater generality, namely for general compact symmetric spaces. The diagonalising transform will then have sequences  as its image, indexed by the irreducible unitary representations of the group occurring in the Plancherel formula for the symmetric space, and taking values in the operators on the corresponding Hilbert spaces. Such a situation is covered by the general formulation in Section~\ref{subsec:frameworkandresults}.
\end{remark}

\subsection{Jacobi operators}\label{subsec:Jacobi}
Let $\D=C^\infty[-1,1]$, carrying its usual Fr\'echet topology generated by the seminorms  $p_k$, for $k\in\N_0$, where, for $g\in\D$, $p_k(g)=\max_{t\in[-1,1]}|g^{(k)}(t)|$. We let $\Ddual$ denote the continuous dual of $\D$. As in Section~\ref{subsec:Fourier}, for $\dist\in\Ddual$, and $g \in\D$, we write $\langle\dist, g\rangle$ for the canonical pairing.
Let $\alpha,\beta > -1$, and consider the weight function
\begin{equation}\label{eq:weightJacobi}
w_{(\alpha,\beta)}(t)=(1-t)^\alpha(1+t)^\beta\qquad(t\in(0,1)).
\end{equation}
We identify $\D$ with a subspace of $\Ddual$ by letting $f\in\D$ act on $\D$ via
\[
\langle f,g\rangle=\int_{-1}^1 f(t)g(t)w_{(\alpha,\beta)}(t)\,dt\qquad(g\in\D).
\]

We will be concerned with the Jacobi-type differential operator
\begin{equation}\label{eq:operatorJacobi}
T_{(\alpha,\beta)}=(1-t^2)\frac{d^2}{dt^2}+[\beta-\alpha-(\alpha+\beta+2)t]\frac{d}{dt}
\end{equation}
with domain $\D$. Obviously, $\D$ is invariant under $T_{(\alpha,\beta)}$, and $T:\D\to\D$ is continuous. A routine computation yields that, for $f,g,\in\D$,
\begin{equation}\label{eq:symmetryJacobi}
\langle T_{(\alpha,\beta)}f,g\rangle=\langle f,T_{(\alpha,\beta)g}\rangle.
\end{equation}
Since $T_{(\alpha,\beta)}:\D\to\D$ is continuous, we can extend (in view of \eqref{eq:symmetryJacobi}) $T_{(\alpha,\beta)}$ from the embedded copy of $\D$ in $\D^\prime$ to $\D^\prime$ by
\begin{equation}\label{eq:extensiontodistributionsJacobi}
\langle T_{(\alpha,\beta)}\dist,g\rangle=\langle \dist, T_{(\alpha,\beta)g }\rangle\qquad(\dist\in\Ddual,\,g\in\D).
\end{equation}

We will use normalised Jacobi polynomials to diagonalise $T$ on $\D$, in fact on $\Ddual$.
For $n\in\N_0$, let
\begin{equation}\label{eq:definitionnormalisedJacobi}
p_n^{(\alpha,\beta)}(t)=\left(\frac{(2n+\alpha+\beta+1)\Gamma(n+\alpha+\beta+1)n!}{2^{\alpha+\beta+1}\Gamma(n+\alpha+1)\Gamma(n+\beta+1)}\right)^{\frac{1}{2}}P_n^{(\alpha,\beta)}(t)
\end{equation}
be the normalised Jacobi polynomial, where the Jacobi polynomial $P_n^{(\alpha,\beta)}$ of order $n$ is, as in \cite[22.2.1]{AbrSte}, \cite[10.8.(3)]{Erd}, and \cite[(4.1.1)]{Sze}, standardised by
\[
P_n^{(\alpha,\beta)}(1)={\binom{n+\alpha}{n}}.
\]
Then, for $n \in\N_0$, $p_n^{(\alpha,\beta)}$ is a real-valued element of $\D$, and, by \cite[22.6.1]{AbrSte}, \cite[10.8.(14)]{Erd}, or \cite[(4.2.1)]{Sze},
\begin{equation}\label{eq:eigenvalueJacobi}
T_{(\alpha,\beta)}p_n^{(\alpha,\beta)}=-n(n+\alpha+\beta+1)p_n^{(\alpha,\beta)}.
\end{equation}
For all $\alpha,\beta>-1$, $\{p_n^{(\alpha,\beta)}\}_{n=0}^\infty$ is an orthonormal basis of $L^2([-1,1],w_{(\alpha,\beta)}\,dt)$, cf.\ \cite[22.2.1]{AbrSte}, or \cite[(4.3.4)]{Sze}.

Let $\S(\N_0)$ be the space of rapidly decreasing sequences on $\N_0$, in its usual Fr\'echet topology, and define $\F_{(\alpha,\beta)}:\D\to \S(\N_0)$ by
\begin{align}\label{eq:transformfunctionsJacobi}
\F_{(\alpha,\beta)}(f)(n)&=\langle f,p_n^{(\alpha,\beta)}\rangle\\
&=\int_{-1}^1 f(t) p_n^{(\alpha,\beta)}(t)w_{(\alpha,\beta)}(t),dt\qquad(f\in\D,\,n\in\N_0).\notag
\end{align}
As in the previous example, it follows from \eqref{eq:symmetryJacobi} and \eqref{eq:eigenvalueJacobi} that the image of an element of $\D$ under $\F_{(\alpha,\beta)}$ is, in fact, an element of $\S(\N_0)$. Moreover, by \cite[Theorem~3.2]{GlaRun}, if $\alpha,\beta\geq-\frac{1}{2}$, then $\F_{(\alpha,\beta)}$ is an isomorphism of topological vector spaces between $\D$ and $\S(\N_0)$, and, for $f\in\D$, the series
\begin{equation}\label{eq:seriesJacobi}
\sum_{n=0}^\infty \F_{(\alpha,\beta)}(f)(n)p_n^{(\alpha,\beta)}
\end{equation}
converges to $f$ in the topology of $\D$.\footnote{The topological isomorphism between $\D$ and $\S(\N_0)$ in \cite{GlaRun} is defined in terms of polynomials with a normalisation other than ours. Since, as a consequence of Stirling's formula, that normalisation and ours differ by a slowly increasing sequence in $n$, the transform as defined with our normalisation is a topological isomorphism as well. The series in \cite[Theorem~3.2]{GlaRun} coincides with ours in \eqref{eq:seriesJacobi}.}

We can extend $\F_{(\alpha,\beta)}$ from $\D$ to $\Ddual$, obtaining an injective map $\F:\Ddual\to\prod_{n\in\N_0}\C$, by defining
\[
\F_{(\alpha,\beta)}(\dist)(n)=\langle \dist,p_n^{(\alpha,\beta)}\rangle \qquad (\dist\in \Ddual,\, n\in\N_0).
\]
As in \eqref{eq:diagonalisationFourier}, we find that
\[
\F_{(\alpha,\beta)} (T_{(\alpha,\beta)} \dist)(n)=-n(n+\alpha+\beta+1)\F\dist(n)\qquad(\dist\in\Ddual,\,n\in\N_0).
\]

We will now apply the results in Section~2, in a similar manner as in Section~\ref{subsec:Fourier}, but with an extra ingredient.  Let $X=L^p([-1,1],w_{(\alpha,\beta)}(t)\,dt)$, for some fixed $1\leq p\leq\infty$. We view $\D$ as a subspace of $X$, so that $T_{(\alpha,\beta)}$ is an operator on $X$ with domain $\D$. We embed $X$ as an abstract vector subspace into $\Ddual$ via the pairing
\[
\langle f,g \rangle=\int_{-1}^1 f(t)g(t)w_{(\alpha,\beta)}(t)\,dt\qquad(f\in X,\, g\in\D).
\]
Let $T_{(\alpha,\beta),c}$ be the operator on $X$ with domain $\D_c$, consisting of those $f\in X$ such that the distribution $T_{(\alpha,\beta)}f$, as defined in \eqref{eq:extensiontodistributionsJacobi}, is in $X$, and defined, for such $f$, by $T_{(\alpha,\beta), c} f=T_{(\alpha,\beta)}f$. Then $T_{(\alpha,\beta),c}$ is a closed extension of $T_{(\alpha,\beta)}$.

Let $T_e$ be an operator on $X$, such that
\begin{equation}\label{eq:extensionJacobi}
T_{(\alpha,\beta)}\subset T_e\subset T_{(\alpha,\beta),c}.
\end{equation}
As in Section~\ref{subsec:Fourier}, one verifies easily that the basic assumption and the assumptions (A1) and (A3) are satisfied.  As to assumption (A2), however, it now needs proof that, for $f\in\D$, the series in \eqref{eq:seriesJacobi} converges absolutely in $X$, and with sum $f$. For the Fourier series in the previous example this was immediately clear since $\Vert e_n\Vert_p=1$ $(n\in\Z, 1\leq p\leq\infty)$, but here we need the following additional argument. By \cite[22.14.1]{AbrSte}, or \cite[(7.32.2)]{Sze}, we have the estimate
\begin{equation}\label{eq:estimateJacobi}
|P_n^{(\alpha,\beta)}(t)|\leq {\binom{n+\max(\alpha,\beta)}{n}}\qquad(n\in\N_0,\,t\in[-1,1]),)
\end{equation}
for all $\alpha,\beta>-1$ such that $\max(\alpha,\beta)\geq -\frac{1}{2}$.\footnote{A different upper bound ${\binom{n+\max(\alpha,\beta)-1}{n}}$ is given in \cite[10.18.(12)]{Erd}. While this seems to be a misprint, this upper bound would also fit into our argument.}
For such $\alpha$ and $\beta$, hence in particular when $\alpha,\beta\geq-\frac{1}{2}$, the combination of \eqref{eq:definitionnormalisedJacobi}, \eqref{eq:estimateJacobi}, and Stirling's formula shows that the sequence $\{\max_{t\in[-1,1]}|p_n^{(\alpha,\beta)}(t)|\}_{n=0}^\infty$ is slowly increasing. The same is then true for the sequence of norms $\{\Vert p_n^{(\alpha,\beta)}\Vert\}_{n=0}^\infty$ in $X=L^p([-1,1],w_{(\alpha,\beta)}(t)\,dt)$. Since the coefficients $\F_{(\alpha,\beta)}(f)(n)$ in \eqref{eq:seriesJacobi} form a rapidly decreasing sequence whenever $f\in\D$, this implies that, for such $f$, the series in \eqref{eq:seriesJacobi} is absolutely convergent in $X$, as required. That the series converges to $f$ in $X$ follows, as in the previous example, from the fact that this is the case in $\D$, combined with the continuity of the inclusion of $\D$ into $X$. Hence assumption (A2) is also satisfied after all, and we have the following result.

\begin{theorem}\label{thm:summarizingtheoremJacobi}
Let $1\leq p\leq\infty$, let $\alpha,\beta\geq -\frac{1}{2}$, and let $w_{(\alpha,\beta)}$ and $T_{(\alpha,\beta)}$ be as in \eqref{eq:weightJacobi} and \eqref{eq:operatorJacobi}, respectively. Let $\D=C^\infty[-1,1]$. View $T$ as an operator on $L^p([-1,1],w_{(\alpha,\beta)}\,dt)$ with domain $\D$. Let $T_e$ be an extension of $T_{(\alpha,\beta)}$ on $L^p([-1,1],w_{(\alpha,\beta)}\,dt)$ as in \eqref{eq:extensionJacobi}. Then $T_e$ has the single-valued extension property.

Let $f\in\D$, and let $\F_{(\alpha,\beta)} f:\N_0\to\C$ be the Jacobi transform of $f$ as in \eqref{eq:transformfunctionsJacobi}.  Then
\begin{equation}\label{eq:localspectralradiusformulaJacobi}
\lim_{n\to\infty}\Vert T_e^n f\Vert_p^{1/n}=\sup\left\{|z| : z\in \{-n(n+\alpha+\beta+1) : n\in\N_0,\, \F_{(\alpha,\beta)} f(n)\neq 0\}\right\}
\end{equation}
in the extended positive real numbers. Moreover,
\[
\sigma_{T_e}(f)=\{-n(n+\alpha+\beta+1) : n\in\N_0,\, \F_{(\alpha,\beta)} f(n)\neq 0\},
\]
so that \eqref{eq:localspectralradiusformulaJacobi} is a local spectral radius formula for $T_e$ at $f$.
\end{theorem}

\subsection{Hermite operator}\label{subsec:Hermite}
Let $\D=\S(\R)$, the Schwartz space of rapidly decreasing functions on $\R$, carrying its Fr\'echet topology generated by the seminorms $p_{k,n}$, for $k,n\in\N_0$, where, for $g\in\S(\R)$, $p_{k,n}(g)=\sup_{t\in[0,\infty)}t^k|g^{(n)}(t)|$. Its dual $\Ddual$ consists of the tempered distributions on $\R$, and we write $\langle\dist,g\rangle$ for the canonical pairing between $\dist\in\Ddual$ and $g\in\D$. We identify $\D$ with a subspace of $\Ddual$, by letting $f\in\D$ act on $\D$ via
\[
\langle f,g\rangle=\int_{\R}f(t)g(t)\,dt\qquad(g\in\D).
\]

Consider the Hermite operator
\begin{equation}\label{eq:operatorHermite}
\frac{d^2}{dt^2}-t^2.
\end{equation}
Obviously, $\D$ is invariant under $T$, and $T:\D\to\D$ is continuous. It is virtually immediate that, for $f,g\in\D$,
\begin{equation}\label{eq:symmetryHermite}
\langle Tf,g\rangle=\langle f,Tg\rangle.
\end{equation}
We extend $T$ from $\D\subset\Ddual$ to $\Ddual$ by
\begin{equation}\label{eq:extensiontodistributionsHermite}
\langle T\dist,g\rangle=\langle \dist,Tg\rangle.
\end{equation}

For $n\in\N_0$, and $t\in\R$, let, as in \cite[p.142]{ReeSim}, or \cite[(1.1.2) and (1.1.18)]{Tha},
\begin{align*}
h_n(t)&=(-1)^n\left( 2^n n! \sqrt{\pi} \right)^{-\frac{1}{2}}e^{\frac{1}{2}t^2}\frac{d^n}{dt^n}e^{-t^2}\\
&=\left( 2^n n! \sqrt{\pi} \right)^{-\frac{1}{2}} H_n(t)e^{-\frac{1}{2}t^2},
\end{align*}
where
\[
H_n(t)=(-1)^n e^{t^2}\frac{d^n}{dt^n}e^{-t^2}\qquad(t\in\R)
\]
is the Hermite polynomial of order $n$, cf.\ \cite[22.11.7]{AbrSte}, \cite[10.13.(7)]{Erd}, or \cite[(1.1.1]{Tha}. Then, for $n\in\N_0$, $h_n$ is a real-valued element of $\D$, and, by \cite[22.6.20]{AbrSte}, \cite[10.13.(13)]{Erd}, or \cite[p.142]{ReeSim},
\begin{equation}\label{eq:eigenvalueHermite}
Th_n=-(2n+1) h_n.
\end{equation}
Furthermore, $\{h_n\}_{n=0}^\infty$ is an orthonormal basis of $L^2(\R,dx)$; cf.\ \cite[22.2.14]{AbrSte}, \cite[(5.5.1)]{Sze} for orthonormality, and \cite[(5.7.2)]{Sze} for completeness.

Let $\S(\N_0)$ be the space of rapidly decreasing sequences on $\N_0$, in its usual Fr\'echet topology, and define $\F:\D\to \S(\N_0)$ by
\begin{align}\label{eq:transformfunctionsHermite}
\F(f)(n)&=\langle f,h_n\rangle\\
&=\int_\R f(t)h_n(t)\,dt \qquad (f\in \D,\, n\in\N_0).\notag
\end{align}
Then \eqref{eq:symmetryHermite} and \eqref{eq:eigenvalueHermite} imply that the image of an element of $\D$ under $\F$ is, in fact, an element of $\S(\N_0)$. Moreover, by \cite[Theorem~V.13]{ReeSim} and its proof, see also \cite[p.262]{Schw}, $\F$ is an isomorphism of topological vector spaces between $\D$ and $\S(\N_0)$, and, for $f\in\D$, the series
\begin{equation}\label{eq:seriesHermite}
\sum_{n=0}^\infty \F(f)(n)h_n
\end{equation}
converges to $f$ in the topology of $\D$.

We can extend $\F$ from $\D$ to $\Ddual$, obtaining an injective map $\F:\Ddual\to\prod_{n\in\N_0}\C$, by defining
\[
\F(\dist)(n)=\langle \dist,h_n\rangle \qquad (\dist\in \Ddual, \, n\in\N_0).
\]
As in \eqref{eq:diagonalisationFourier}, we find that
\[
\F (T \dist)(n)=-(2n+1)h_n.
\]

We can now apply the results in Section~2, similarly as in the previous examples. Let $X=L^p(\R,dt)$. We view $\D$ as a subspace of $X$, and we embed $X$ into $\Ddual$ via the pairing
\[
\langle f,g \rangle=\int_{\R} f(t)g(t)\,dt\qquad(f\in X,\, g\in\D).
\]
We let $T$ be the operator on $X$ with domain $\D_c$, consisting of those $f\in X$ such that the distribution $Tf$, as defined in \eqref{eq:extensiontodistributionsHermite}, is in $X$, and defined, for such $f$, by $T_cf=Tf$. Then $T_c$ is a closed extension of $T$. Let $T_e$ be an operator on $X$, such that
\begin{equation}\label{eq:extensionHermite}
T\subset T_e\subset T_c.
\end{equation}

As in Section~\ref{subsec:Jacobi}, one verifies easily that the basic assumption and the assumptions (A1) and (A3) are satisfied, whereas it needs proof that, for $f\in\D$, the series in \eqref{eq:seriesHermite} converges absolutely in $X$, and with sum $f$. To verify this, we note that the sequence of norms $\{\Vert h_n\Vert\}_{n=0}^\infty$ in $X=L^p([0,\infty),dt)$ is slowly increasing, by \cite[Lemma~1.5.2]{Tha}.
Since the coefficients $\F(f)(n)$ in \eqref{eq:seriesHermite} form a rapidly decreasing sequence whenever $f\in\D$, this implies, as before, that, for such $f$, the series in \eqref{eq:seriesHermite} is absolutely convergent in $L^p(\R,dt)$, as required. As before, it converges to $f$ in $X$ because it does so in $\D$, and the inclusion of $\D$ into $X$ is continuous. Hence assumption (A2) is satisfied again, and we have the following result.

\begin{theorem}\label{thm:summarizingtheoremHermite}
Let $1\leq p\leq\infty$. Let $\D=\S(\R)$. Let $T$ be as in \eqref{eq:operatorHermite}, and view $T$ as an operator on $L^p(\R,dt)$ with domain $\D$. Let $T_e$ be an extension of $T$ on $L^p(\R,dt)$ as in \eqref{eq:extensionHermite}. Then $T_e$ has the single-valued extension property.

Let $f\in\D$, and let $\F f:\N_0\to\C$ be the Hermite transform of $f$ as in \eqref{eq:transformfunctionsHermite}. Then
\begin{equation}\label{eq:localspectralradiusformulaHermite}
\lim_{n\to\infty}\Vert T_e^n f\Vert_p^{1/n}=\sup\left\{|z| : z\in \{-(2n+1) : n\in\N_0,\, \F f(n)\neq 0\}\right\}
\end{equation}
in the extended positive real numbers. Moreover,
\[
\sigma_{T_e}(f)=\{-(2n+1) : n\in\N_0,\, \F f(n)\neq 0\},
\]
so that \eqref{eq:localspectralradiusformulaHermite} is a local spectral radius formula for $T_e$ at $f$.
\end{theorem}

\begin{remark}
Actually, the expansion with respect to the product of one-variable Hermite functions provides a topological isomorphism between $\S(\R^d)$ and the rapidly decreasing sequences on $\N_0^n$ for arbitrary $d=1,2,\ldots$, according to \cite[Theorem~V.13]{ReeSim}. Hence the results in the present example generalise easily to arbitrary dimension.
 \end{remark}

\subsection{Laguerre operators}\label{subsec:Laguerre}
Let $\S^+$ consist of the restrictions of the elements of $\S(\R)$ to $[0,\infty)$, supplied with a Fr\'echet topology by means of the seminorms $p_{k,n}$, for $k,n\in\N_0$, where, for $g\in\S^+$, $p_{k,n}(g)=\sup_{t\in[0,\infty)}t^k|g^{(n)}(t)|$. For $\alpha>-1$, let $\D_\alpha=t^{\alpha/2}\S^+=\{t^{\alpha/2}f : f\in\S^+\}$, and transport the Fr\'echet topology from $\S^+$ to  $\D_\alpha$ via the canonical linear bijection with $\S^+$. As before, we let $\Ddual_\alpha$ be the continuous dual of $\D_\alpha$, and $\langle\dist,g\rangle$ denotes the canonical pairing between $\dist\in\Ddual_\alpha$ and $f\in\D_\alpha$. Since $\D_\alpha$ consists of locally integrable functions, we can view $\D_\alpha$ as a subspace of $\Ddual_\alpha$, by letting $f\in\D_\alpha$ act on $\D_\alpha$ via
\[
\langle f,g\rangle=\int_{0}^\infty f(t)g(t)\,dt\qquad(g\in\D_\alpha).
\]

Consider the Laguerre-type operator
\begin{equation}\label{eq:operatorLaguerre}
T_\alpha=t\frac{d^2}{dt^2} + \frac{d}{dt} - \frac{t}{4}-\frac{\alpha^2}{4t}+\frac{\alpha+1}{2}
\end{equation}
with domain $\D_\alpha$. It is routine to check that $T_\alpha$ leaves $\D_\alpha$ invariant, that $T_\alpha:\D_\alpha\to\D_\alpha$ is continuous, and that, for $f,g\in\D_\alpha$,
\begin{equation*}
\langle T_\alpha f,g\rangle=\langle f, T_\alpha g\rangle.
\end{equation*}
We extend $T_\alpha$ from $\D_\alpha\subset\Ddual_\alpha$ to $\Ddual_\alpha$ by
\begin{equation}\label{eq:extensiontodistributionsLaguerre}
\langle T_\alpha\dist,g\rangle=\langle\dist,T_\alpha g\rangle\qquad(\dist\in\Ddual_\alpha,\,g\in\D_\alpha).
\end{equation}

For $n\in\N_0$, let, as in \cite[(1.1.44)]{Tha},
\[
\Lag_n^\alpha(t)=\left(\frac{n!}{\Gamma(n+\alpha+1)}\right)^{\frac{1}{2}}e^{-\frac{t}{2}}t^{\frac{\alpha}{2}}L_n^\alpha(t)\qquad(t> 0)
\]
denote the normalised generalised Laguerre function, where $L_n^\alpha$ is the generalised Laguerre polynomial of order $n$, as in \cite[22.11.6]{AbrSte} or \cite[(1.1.37)]{Tha}.
Then, for $n\in\N_0$, $\Lag_n^\alpha$ is a real-valued element of $\D_\alpha$, and, by \cite[10.12.(11)]{Erd},
\[
T_\alpha \Lag_n^\alpha=-n\Lag_n^\alpha.
\]
For each $\alpha>-1$, $\{\Lag_n^\alpha\}_{n=0}^\infty$ is an orthonormal basis of $L^2([0,\infty),dt)$, see \cite[22.2.12]{AbrSte} or \cite[(1.1.44)]{Tha} for orthonormality, and \cite[5.7.1]{Sze} for completeness.

Let $\S(\N_0)$ be the space of rapidly decreasing sequences on $\N_0$, topologised in the usual way, and define $\F_\alpha:\D_\alpha\to \S(\N_0)$ by
\begin{align}\label{eq:transformfunctionsLaguerre}
\F_\alpha(f)(n)&=\langle f,\Lag_n^\alpha\rangle\\
&=\int_0^\infty f(t)\Lag_n^\alpha(t)\,dt\qquad(f\in\D_\alpha,\,n\in\N_0).\notag
\end{align}
By \cite[Theorem~2.5]{Dur}, $\F_\alpha$ is an isomorphism of topological vector spaces between $\D_\alpha$ and $\S(\N_0)$, and, for $f\in\D_\alpha$, the series
\begin{equation}\label{eq:seriesLaguerre}
\sum_{n=0}^\infty \F_\alpha(f)(n)\Lag_n^\alpha
\end{equation}
converges to $f$ in the topology of $\D_\alpha$.

We extend $\F$ from $\D_\alpha$ to $\Ddual_\alpha$, obtaining an injective map $\F:\Ddual_\alpha\to\prod_{n\in\N_0}\C$, by defining
\[
\F(\dist)(n)=\langle \dist,\Lag_n^\alpha\rangle\qquad(\dist\in\Ddual_\alpha,\,n\in\N_0),
\]
and then
\[
\F(T_\alpha\dist)(n)=-n\F\dist(n)\qquad(\dist\in\Ddual,\,n\in\N_0).
\]

In this example, we will, for $\alpha>-1$ and $1\leq p\leq\infty$, study extensions of the operator $T_\alpha$ on $L^p([0,\infty),dt)$ with original domain $\D_\alpha$, in the situation where actually $\D_\alpha\subset L^p([0,\infty),dt)$. Since $\alpha$ can be negative, this inclusion is not automatic. Furthermore, assumption (A1) necessitates us to require that $\Lag_n^{\alpha}\in L^q([0,\infty),dt)$ ($n\in\N_0$), where $q$ is the conjugate exponent of $p$. This yields another condition on $p$ and $\alpha$. A routine verification shows that, for $\alpha>-1$ and $1\leq p\leq\infty$, the requirement that $\D_\alpha\subset L^p([0,\infty),dt)$ and that simultaneously $\Lag_n^\alpha\in L^q([0,\infty),dt)$ ($n\in\N_0$), is satisfied precisely when \begin{equation}\label{eq:conditionsLaguerre}
\begin{cases}1\leq p\leq\infty&\textup{if }\alpha\geq 0;\\
\frac{2}{\alpha+2}<p<-\frac{2}{\alpha}&\textup{if } -1<\alpha<0.
\end{cases}
\end{equation}
Assuming that \eqref{eq:conditionsLaguerre} is satisfied, we let $X=L^p(\R,dt)$. We view $\D_\alpha$ as a subspace of $X$, and we embed $X$ into $\Ddual_\alpha$ via the pairing
\[
\langle f,g \rangle=\int_{\R} f(t)g(t)\,dt\qquad(f\in X,\, g\in\D_\alpha).
\]
We let $T_{\alpha,c}$ be the operator on $X$ with domain $\D_{\alpha,c}$, consisting of those $f\in X$ such that the distribution $T_\alpha f$, as defined in \eqref{eq:extensiontodistributionsLaguerre}, is in $X$, and defined, for such $f$, by $T_{\alpha,c}f=T_\alpha f$. Then $T_{\alpha, c}$ is a closed extension of $T_\alpha$. Let $T_e$ be an operator on $X$, such that
\begin{equation}\label{eq:extensionLaguerre}
T_\alpha\subset T_e\subset T_{\alpha,c}.
\end{equation}
As in the previous two examples, all assumptions are easily seen to be satisfied, except for assumption (A2) again. To this end, we note that, for all $\alpha>-1$ and $1\leq p<\infty$ such that $\alpha p>-2$, the sequence $\{\Vert \Lag_n^\alpha\Vert_p\}_{n=0}^\infty$ is slowly increasing. This follows from \cite[Lemma~1.5.4]{Tha} for finite $p\neq 4$ and from \cite[Lemma~1.5.3]{Tha} for $p=4$. For $p=\infty$ and $\alpha\geq 0$, this is also true, as follows from \cite[Lemma~1.5.3]{Tha}. In particular, $\{\Vert \Lag_n^\alpha\Vert_p\}_{n=0}^\infty$ is slowly increasing whenever \eqref{eq:conditionsLaguerre} are satisfied. As in the previous examples, combining this with the fact that the coefficients $\F_\alpha(f)(n)$ in \eqref{eq:seriesLaguerre} form a rapidly decreasing sequence whenever $f\in\D_\alpha$ shows that, for such $f$, the series in \eqref{eq:seriesLaguerre} is absolutely convergent in $X$, and with sum $f$, as required. Hence we have the following.

\begin{theorem}\label{thm:summarizingtheoremLaguerre}
Let $\alpha>-1$ and $1\leq p\leq\infty$ satisfy \eqref{eq:conditionsLaguerre}. Let $\D_\alpha=t^{\alpha/2}\S^+$ as above. Let $T_\alpha$ be as in \eqref{eq:operatorLaguerre}, and view $T_\alpha$ as an operator on $L^p(\R,dt)$ with domain $\D$. Let $T_e$ be an extension of $T_\alpha$ on $L^p(\R,dt)$ as in \eqref{eq:extensionLaguerre}. Then $T_e$ has the single-valued extension property.

Let $f\in\D_\alpha$, and let $\F_\alpha f:\N_0\to\C$ be the Laguerre transform of $f$ as in \eqref{eq:transformfunctionsLaguerre}. Then
\begin{equation}\label{eq:localspectralradiusformulaLaguerre}
\lim_{n\to\infty}\Vert T_e^n f\Vert_p^{1/n}=\sup\left\{|z| : z\in \{-n : n\in\N_0,\, \F_\alpha f(n)\neq 0\}\right\}
\end{equation}
in the extended positive real numbers. Moreover,
\[
\sigma_{T_e}(f)=\{-n : n\in\N_0,\, \F_\alpha f(n)\neq 0\},
\]
so that \eqref{eq:localspectralradiusformulaLaguerre} is a local spectral radius formula for $T_e$ at $f$.
\end{theorem}

\subsection{A class of examples on bounded domains}\label{subsec:classofexamples}
In \cite{Zer} the following result is stated: Let $\Omega$ be a bounded domain in $\R^d$ with a Lipschitzian boundary, and let $C^\infty(\overline \Omega)$ denote the space of all infinitely differentiable functions on the closure of $\Omega$, in its usual Fr\'echet topology. Let $w$ be a weight function in $L^1(\Omega,dt)$, such that $w(t)\geq 1$, for $t\in\Omega$. Let $\{p_n\}_{n=0}^\infty$ be an orthonormal basis of $L^2(\Omega, w(t)\,dt)$, consisting of polynomials of a degree which is non-decreasing in $n$. Then the operation of taking Fourier coefficients of $f\in C^\infty(\overline\Omega)$ with respect to the orthonormal basis $\{p_n\}_{n=0}^\infty$ establishes a topological isomorphism between $C^\infty(\overline\Omega)$ and the space $\S(\N_0)$ of rapidly decreasing sequences.\footnote{It is not unusual for spaces of test functions to be topologically isomorphic to (sub)spaces of rapidly decreasing sequences, as is attested by the material in, e.g., \cite{Trie}, \cite{Val}, and \cite{Vog}. However, the precise statement that, for a domain and weight of the type under consideration, taking Fourier coefficients with respect to a polynomial orthonormal basis as indicated yields such an isomorphism, seems to be less widely known than it deserves. Possibly this is related to the fact that the actual detailed proof seems to be only available in French in the local publication \cite{Pav}, which is not at the disposal of the current authors.}

It is obvious from our previous examples that such a topological isomorphism is the non-automatic key ingredient for an application of our results in concrete situations. Hence the result from \cite{Zer} implies a range of examples, as follows: Suppose that $\{p_n\}_{n=0}^\infty$ is an orthonormal basis as indicated. We can assume that these polynomials are real-valued. Let $\D=C^\infty(\overline\Omega)$, with continuous dual $\Ddual$. Suppose that $T$ is a symmetric operator on $L^2(\Omega,w(t)\,dt)$ with invariant domain $C^\infty(\overline\Omega)$, such that $T:\D\to\D$ is continuous, commutes with pointwise conjugation, and has the $p_n$ as eigenfunctions, with (real) eigenvalues $\eig(n)$  $(n\in\N_0)$. We embed $\D$ into $\Ddual$ by letting $f\in\D$ act on $\D$ via
\[
\langle f,g\rangle=\int_\Omega f(t)g(t)w(t)\,dt\qquad(g\in\D).
\]
Then the assumptions on $T$ imply that, for $f,g\in\D$,
\[
\langle Tf,g\rangle=\langle f,Tg\rangle,
\]
so that we can extend $T$ to $\Ddual$ by defining
\begin{equation}\label{eq:extensiontodistributionsZerner}
\langle T\dist,g\rangle=\langle \dist,Tf\rangle\qquad(g\in\D).
\end{equation}
We define $\F:\D\to\S(\N_0)$ (!) by
\begin{align}\label{eq:transformfunctionsZerner}
\F f(n)&=\langle f,p_n\rangle\\
&=\int_\Omega f(t)p_n(t)w(t)\,dt\qquad(f\in\D,\,n\in\N_0),\notag
\end{align}
which, since the $p_n$ are real-valued, is taking the sequence of Fourier coefficients, and extend $\F$ to $\Ddual$ by
\[
(\F\dist)(n)=\langle \dist,p_n\rangle\qquad(\dist\in\Ddual,\,n\in\N_0).
\]
Suppose $1\leq p\leq\infty$. Let $X=L^p(\Omega,w(t)\,dt)$, and view $T$ as an operator on $X$ with domain $\D$. We embed $X$ into $\Ddual$ by defining
\[
\langle f,g\rangle=\int_\Omega f(t)g(t)w(t)\,dt\qquad(f,\in X,\,g\in\D).
\]
As before, we let $T_c$ denote the operator on $X$ with domain $\D_c$, consisting of those $f\in X$ such that the distribution $Tf$, as defined in \eqref{eq:extensiontodistributionsZerner}, is in $X$, and defined, for such $f$, by $T_cf=Tf$. Let $T$ be an operator on $X$, such that
\begin{equation}\label{eq:extensionZerner}
T\subset T_e\subset T_c.
\end{equation}
It is then easily verified that, for $p=2$, the basic assumption and the assumptions (A1), (A2), and (A3) are satisfied. The absolute convergence of the series in (A2) for $f\in\D$ is then automatic, since $\Vert p_n\Vert_2=1$ ($n\in\N_0$). Hence, for $p=2$, results as in the previous examples hold for $T_e$ and $f\in\D$. Actually, however, these results hold for arbitrary $1\leq p\leq\infty$. Indeed, it is not too difficult to see that the fact that $\F:\D\to\S(\N_0)$ is a topological isomorphism implies that the sequence $\{\max_{t\in\Omega}|p_n(t)|\}_{n=0}^\infty$ is necessarily slowly increasing. Hence the same holds for the sequence of norms $\{\Vert p_n\Vert\}_{n=0}^\infty$ in $X=L^p(\Omega,w(t)\,dt)$, and then, as in the previous examples, one sees that assumption (A2) is satisfied as well. Therefore, we have the following, when combining our results with the ones stated in \cite{Zer}.

\begin{theorem}\label{thm:summarizingtheoremZerner}
Let $\Omega$ be a bounded domain in $\R^d$ with Lipschitzian boundary. Let $w(t)\geq 1$ be a weight function on $\Omega$. Let $\{p_n\}_{n=0}^\infty$ be an orthonormal basis of $L^2(\Omega, w(t)\,dt)$, consisting of real-valued polynomials of a degree which is non-decreasing in $n$. Let $\D=C^\infty(\overline\Omega)$, supplied with its usual Fr\'echet topology. Suppose that $T:\D\to\D$ is continuous and commutes with pointwise conjugation, and that it is a symmetric operator on $L^2(\Omega,w(t)\,dt)$. Furthermore, assume that, for $n\in\N_0$, $Tp_n=\eig(n)p_n$ \textup{(}where the eigenvalues $\eig(n)$ are necessarily real\textup{)}.

Let $1\leq p\leq\infty$, and view $T$ as an operator on $L^p(\Omega,w(t)\,dt)$ with domain $\D$. Let $T_e$ be an extension of as in \eqref{eq:extensionZerner}. Then $T_e$ has the single-valued extension property.

Let $f\in\D$, and let $\F f:\N_0\to\C$ be the transform of $f$ as in \eqref{eq:transformfunctionsZerner}. Then
\begin{equation}\label{eq:localspectralradiusformulaZerner}
\lim_{n\to\infty}\Vert T_e^n f\Vert_p^{1/n}=\sup\left\{|z| : z\in \{\eig(n) : n\in\N_0,\, \F f(n)\neq 0\}^\cl\right\}
\end{equation}
in the extended positive real numbers. Moreover,
\[
\sigma_{T_e}(f)=\{\eig(n) : n\in\N_0,\, \F f(n)\neq 0\}^\cl,
\]
so that \eqref{eq:localspectralradiusformulaZerner} is a local spectral radius formula for $T_e$ at $f$.
\end{theorem}

\begin{remark}
The preceding examples do not overlap with Theorem~\ref{thm:summarizingtheoremZerner}, except the case $-\frac{1}{2}\leq \alpha,\beta\leq 0$ for Jacobi operators in Section~\ref{subsec:Jacobi}.
\end{remark}

\end{document}